\documentclass[dvips]{imsart}

\usepackage{amsthm,amsmath,natbib}
\usepackage{graphics,ifthen}
\usepackage{latexsym}
\usepackage{mathrsfs}
\usepackage{amssymb}

\startlocaldefs

\newcommand{\indic}{\mathbb{I}}

\numberwithin{equation}{section} \theoremstyle{plain}
\newtheorem{thm}{Theorem}[section]

\newtheorem{prop}{Proposition}[section]

\newtheorem{rem}{Remark}[section]

\endlocaldefs

\begin{document}

\begin{frontmatter}
\title{New Dirichlet Mean Identities}
\runtitle{Dirichlet Means}
\begin{aug}
\author{\fnms{Lancelot F.} \snm{James}\thanksref{t1}\ead[label=e1]{lancelot@ust.hk}}

\thankstext{t1}{Supported in part by the grant HIA05/06.BM03 of the HKSAR }

\runauthor{Lancelot F. James}

\affiliation{Hong Kong University of Science and
Technology}

\address[a]{Lancelot F. James,\\ The Hong Kong University of Science and
Technology, \\Department of Information and Systems Management,\\
Clear Water Bay, Kowloon, Hong Kong. \printead{e1}.}


\end{aug}

\begin{abstract}
An important line of research is the investigation of the laws of
random variables known as Dirichlet means as discussed in
Cifarelli and Regazzini~\cite{CifarelliRegazzini}. However there
is not much information on inter-relationships between different
Dirichlet means. Here we introduce two distributional operations,
which consist of multiplying a mean functional by an independent
beta random variable and an operation involving an exponential
change of measure. These operations identify relationships between
different means and their densities. This allows one to use the
often considerable analytic work to obtain results for one
Dirichlet mean to obtain results for an entire family of otherwise
seemingly unrelated Dirichlet means. Additionally, it allows one
to obtain explicit densities for the related class of random
variables that have generalized gamma convolution distributions,
and the finite-dimensional distribution of their associated L\'evy
processes. This has implications in, for instance, the explicit
description of Bayesian nonparametric prior and posterior models,
and more generally  in a variety of applications in probability
and statistics involving L\'evy processes. We demonstrate how the
technique applies to several interesting examples
\end{abstract}




\end{frontmatter}
\section{Introduction}
In this work we present two distributional operations which
identify relationships between seemingly different classes of
random variables which are representable as linear functionals of
a Dirichlet process, otherwise known as Dirichlet means.
Specifically the first operation consists of multiplication of a
Dirichlet mean by an independent beta random variable and the
second operation involves an exponential change of measure to the
density of a related infinitely divisible random variable having a
generalized gamma convolution distribution~(GGC). This latter
operation is often referred to in the statistical literature as
\emph{exponential tilting} or in mathematical finance as an
\emph{Esscher transform}. We believe our results add a significant
component to the foundational work of Cifarelli and
Regazzini~\cite{CifarelliRegazzini79, CifarelliRegazzini}. In
particular, our results allow one to use the often considerable
analytic work to obtain results for one Dirichlet mean to obtain
results for an entire family of otherwise seemingly unrelated mean
functionals. It also allows one to obtain explicit densities for
the related class of infinitely divisible random variables which
are generalized gamma convolutions, and the finite-dimensional
distribution of their associated L\'evy processes,(see
Bertoin~\cite{Bertoin} for the formalities of general L\'evy
processes). The importance of this latter statement is that L\'evy
processes now commonly appear in variety of applications in
probability and statistics. A detailed summary and outline of our
results may be found in section~\ref{sec:outline}. Some background
information, and notation, on Dirichlet proceses and Dirichlet
means, their connection with GGC random variables, recent
references and some motivation for our work is given in the next
section.

\subsection{Background and motivation}
Let $X$ be a non-negative random variable with cumulative
distribution function $F_{X}$. Note furthermore for a measurable
set $C,$ we use the notation $F_{X}(C)$ to mean the probability
that $X$ is in $C.$ One may define a Dirichlet process random
probability measure, see \cite{Freedman} and
\cite{Ferguson73,Ferguson74}, say $P_{\theta},$ on $[0,\infty)$
with total mass parameter $\theta$ and \emph{prior} parameter
$F_{X},$ via its finite dimensional distribution as follows; for
any disjoint partition on $[0,\infty)$, say $(C_{1},\ldots,
C_{k})$, the distribution of the random vector
$(P_{\theta}(C_{1}),\ldots,P_{\theta}(C_{k}))$ is a $k$-variate
Dirichet distribution with parameters $(\theta
F_{X}(C_{1}),\ldots, \theta F_{X}(C_{k})).$  Hence for each $C$,
$$P_{\theta}(C)=\int_{0}^{\infty}\indic(x\in C)P_{\theta}(dx)$$ has a beta distribution with
parameters $(\theta F_{X}(C),\theta (1-F_{X}(C))).$ Equivalently
setting $\theta F_{X}(C_{i})=\theta_{i}$ for $i=1,\ldots,k,$
$$
(P_{\theta}(C_{1}),\ldots, P_{\theta}(C_{k}))
\overset{d}=\left(\frac{G_{\theta_{i}}}{G_{\theta}};
i=1,\ldots,k\right)
$$
where $(G_{\theta_{i}})$ are independent random variables with
gamma$(\theta_{i},1)$ distributions and
$G_{\theta}=G_{\theta_{1}}+\cdots+G_{\theta_{k}}$ has a
gamma$(\theta,1)$ distribution. This means that one can define the
Dirichlet process via the normalization of an independent
increment gamma process on $[0,\infty)$, say
$\gamma_{\theta}(\cdot),$ as
$$
P_{\theta}(\cdot)=\frac{\gamma_{\theta}(\cdot)}{\gamma_{\theta}([0,\infty))}
$$
where $\gamma_{\theta}(C_{i})\overset{d}=G_{\theta_{i}}$ and whose
almost surely finite total random mass is
$\gamma_{\theta}([0,\infty))\overset{d}=G_{\theta}.$ A very
important aspect of this construction is the fact that
$G_{\theta}$ is independent of $P_{\theta},$ and hence any
functional of $P_{\theta}.$ This is a natural generalization of
Lukacs'\cite{Lukacs} characterization of beta and gamma random
variables, whose work is fundamental to what is now referred to as
the beta-gamma algebra, (for more on this, see Chaumont and
Yor~(\cite{Chaumont}, section 4.2)). See also Emery and
Yor~\cite{EmeryYor} for some interesting relationships between
gamma processes, Dirichlet processes and Brownian bridges.

These simple representations and other nice features of the
Dirichlet process have, since the important work of
Ferguson~\cite{Ferguson73,Ferguson74}, contributed greatly
to the relevance and practical utility of the field of
Bayesian non and semi-parametric statistics. Naturally,
owing to the ubiquity of the gamma and beta random
variables, the Dirichlet process also arises in other
areas. One of the more interesting, and we believe quite
important, topics related to the Dirichlet process is the
study of the laws of random variables called Dirichlet mean
functionals, or simply Dirichlet means, which we denote as
$$
M_{\theta}(F_{X})\overset{d}=\int_{0}^{\infty}xP_{\theta}(dx),
$$
initiated in the works of Cifarelli and
Regazzini~\cite{CifarelliRegazzini79, CifarelliRegazzini}.
In~\cite{CifarelliRegazzini} the authors obtained an important
identity for the Cauchy-Stieltjes transform of order $\theta.$
This identity is often referred to as the Markov-Krein identity as
can be seen in for example, Diaconis and
Kemperman~\cite{Diaconis}, Kerov~\cite{Kerov} and Vershik, Yor and
Tsilevich~\cite{Vershik}, where these authors highlight its
importance to, for instance, the study of the Markov moment
problem, continued fraction theory and exponential representation
of analytic functions. This identity is later called the
Cifarelli-Regazzini identity in~\cite{James2005}. Cifarelli and
Regazzini~\cite{CifarelliRegazzini}, owing to their primary
interest, used this identity to then obtain explicit density and
cdf formulae for $M_{\theta}(F_{X}).$ The density formulae may be
seen as Abel type transforms and hence do not always have simple
forms, although we stress that they are still useful for some
analytic calculations. The general exception is the case of
$\theta=1$ which has a nice form. Some examples of works that have
proceeded along these lines are Cifarelli and
Melilli~\cite{CifarelliMelilli}, Regazzini, Guglielmi and di
Nunno~\cite{Regazzini2002}, Regazzini, Lijoi and
Pr\"unster\cite{Regazzini2003}, Hjort and Ongaro~\cite{Hjort},
Lijoi and Regazzini~\cite{Lijoi}, and Epifani, Guglielmi and
Melilli~\cite{Epifani2004, Epifani2006}). Moreover, the recent
work of~Bertoin, Fujita, Roynette and Yor~\cite{BFRY} and James,
Lijoi and Pr\"unster~\cite{JLP}~(see also~\cite{JamesGamma} which
is a preliminary version of this work) show that the study of mean
functionals is relevant to the analysis of phenomena related to
Bessel and Brownian processes. In fact the work of James, Lijoi
and Pr\"unster~\cite{JLP} identifies many new explicit examples of
Dirichlet means which have interesting interpretations.

Related to these last points, Lijoi and
Regazzini~\cite{Lijoi} have highlighted a close connection
to the theory of generalized gamma
convolutions(see~\cite{BondBook}). Specifically, it is
known that a rich sub-class of random variables having
generalized gamma convolutions~(GGC) distributions may be
represented as
\begin{equation}
T_{\theta}\overset{d}=G_{\theta}M_{\theta}(F_{X})\overset{d}=\int_{0}^{\infty}x\gamma_{\theta}(dx).
\label{gammarep}
\end{equation}
We call these random variables GGC$(\theta,F_{X}).$ In additon we
see from~(\ref{gammarep}) that $T_{\theta}$ is a random variable
derived from a weighted gamma process, and hence the calculus
discussed in Lo~\cite{Lo82} and Lo and Weng~\cite{LW} applies. In
general GGC random variables are an important class of infinitely
divisible random variables whose properties have been extensively
studied by~\cite{BondBook} and others. We note further that
although we have written a GGC$(\theta,F_{X})$ random variable as
$G_{\theta}M_{\theta}(F_{X})$ this representation is not unique
and in fact it is quite rare to see $T_{\theta}$ represented in
this way. We will show that one can in fact exploit this
non-uniqueness to obtain explicit densities for $T_{\theta}$ even
when it is not so easy to do so for $M_{\theta}(F_{X}).$ While the
representation $G_{\theta}M_{\theta}(F_{X})$ is not unique it
helps one to understand the relationship between the Laplace
transform of $T_{\theta}$ and the Cauchy-Stieltjes transform of
order $\theta$ of  $M_{\theta}(F_{X}),$ which indeed characterizes
respectively the law of $T_{\theta}$ and $M_{\theta}(F_{X}).$
Specifically, using the independence property of $G_{\theta}$ and
$M_{\theta}(F_{X}),$ leads to, for $\lambda\ge 0,$
\begin{equation}
\mathbb{E}[{\mbox e}^{-\lambda T_{\theta}}]=\mathbb{E}[{(1+\lambda
M_{\theta}(F_{X}))}^{-\theta}]={\mbox
e}^{-\theta\psi_{F_{X}}(\lambda)} \label{CS}
\end{equation}
where \begin{equation}
\psi_{F_{X}}(\lambda)=\int_{0}^{\infty}\log(1+\lambda
x)F_{X}(dx)=\mathbb{E}[\log(1+\lambda X)]. \label{Levy}
\end{equation}
is the \emph{L\'evy exponent} of $T_{\theta}.$ We note that
$T_{\theta}$ and $M_{\theta}(F_{X})$ exist if and only if
$\psi_{F_{X}}(\lambda)<\infty$ for $\lambda>0,$(see for
instance~\cite{DS} and~\cite{BondBook}). The expressions
in~(\ref{CS}) equates with the the identity obtained by
Cifarelli and Regazzini~\cite{CifarelliRegazzini},
mentioned previously.

Despite these interesting results, there is very little work on
the relationship between  different mean functionals. Suppose, for
instance, that for each  fixed value of $\theta>0,$
 $M_{\theta}(F_{X})$ denotes a Dirichlet mean and  $(M_{\theta}(F_{Z_{c}});c>0)$
 denotes a collection of Dirichlet mean random variables indexed
 by a family of distributions $(F_{Z_{c}};c>0 ).$
 Then one can ask the question, for what choices of $X$ and $Z_{c}$ are
 these mean functionals related, and in what sense? In particular,
 one may wish to know how their densities are related. The
 rationale here is that if such a relationship is established, then the effort that
 one puts forth to obtain results such as the explicit density of $M_{\theta}(F_{X}),$
 can be applied to an entire family of Dirichlet means
 $(M_{\theta}(F_{Z_{c}});c>0).$ Furthermore since
 Dirichlet means are associated with GGC random variables this
 would establish relationships between a GGC$(\theta,F_{X})$ random
 variable and a family of GGC$(\theta,F_{Z_{c}})$ random
 variable
 Simple examples are of course the choices $Z_{c}=X+c$ and
 $Z_{c}=c X,$ which, due to the linearity properties of mean functionals, results easily in the identities in law
 $$
 M_{\theta}(F_{X+c})=c+M_{\theta}(F_{X}){\mbox { and }}
 M_{\theta}(F_{c X})=cM_{\theta}(F_{X})
$$

Naturally, we are going to discuss more complex relationships, but
with the same goal. That is, we will identify non-trivial
relationships so that the often considerable efforts that one
makes in the study of one mean functional $M_{\theta}(F_{X})$ can
be then used to obtain more easily results for other mean
functionals, their corresponding GGC random variables and L\'evy
processes. In this paper we will describe two such operations
which we elaborate on in the next subsection.

\subsection{Outline and summary of results}\label{sec:outline}
Section~\ref{sec:prelim} reviews some of the existing formulae for
the density and cdf of Dirichlet means. In Section~\ref{sec:beta},
we will describe the operation of multiplying a mean functional
$M_{\theta\sigma}(F_{X})$ by an independent beta random variable
with parameters $(\theta \sigma,\theta(1-\sigma)),$ say,
$\beta_{\theta\sigma,\theta(1-\sigma)}$ where $0<\sigma<1.$ We
call this operation \emph{beta scaling}. Theorem~\ref{thm1beta}
shows that the resulting random variable
$\beta_{\theta\sigma,\theta(1-\sigma)}M_{\theta\sigma}(F_{X})$ is
again a mean functional but now of order $\theta$. In addition,
the GGC$(\theta\sigma,F_{X})$ random variable $G_{\theta
\sigma}M_{\theta\sigma}(F_{X})$ is equivalently a GGC random
variable of order $\theta.$ Now keeping in mind that tractable
densities of mean functionals of order $\theta=1$ are the easiest
to obtain, Theorem~\ref{thm1beta} shows that by setting
$\theta=1$, the densities of the uncountable collection of random
variables $(\beta_{\sigma,1-\sigma}M_{\sigma}(F_{X}); 0<\sigma\leq
1),$ are all mean functionals of order $\theta=1.$
Theorem~\ref{thm2beta} then shows that efforts used to calculate
the explicit density of any one of these random variables, via the
formulae of~\cite{CifarelliRegazzini}, lead to the explicit
calculation of the densities of all of them. Additionally,
Theorem~\ref{thm2beta} shows that the corresponding GGC random
variables may all be expressed as GGC random variables of order
$\theta=1,$ representable in distribution as
$G_{1}\beta_{\sigma,1-\sigma}M_{\sigma}(F_{X})$. A key point here
is that Theorem~\ref{thm2beta} gives a tractable density for
$\beta_{\sigma,1-\sigma}M_{\sigma}(F_{X})$ without requiring
knowledge of the density of $M_{\sigma}(F_{X}),$ which is usually
expressed in a complicated manner. These results also will yield
some non-obvious integral identities. Furthermore, noting that a
GGC$(\theta,F_{X})$ random variable, $T_{\theta},$ is infinitely
divisible, we associate it with an independent increment process
$(\zeta_{\theta}(t); t\ge 0)$ known as a subordinator~,(a
non-decreasing non-negative L\'evy process), where for each fixed
$t,$
$$
\mathbb{E}[{\mbox e}^{-\lambda
\zeta_{\theta}(t)}]=\mathbb{E}[{\mbox e}^{-\lambda
T_{\theta t}}]={\mbox e}^{-t\theta \psi_{F_{X}}(\lambda)}.
$$
That is, marginally $\zeta_{\theta}(1)\overset{d}=T_{\theta}$ and
$\zeta_{\theta}(t)\overset{d}=\zeta_{\theta
t}(1)\overset{d}=T_{\theta t}.$ In addition, for $s<t,$
$\zeta_{\theta}(t)-\zeta_{\theta}(s)\overset{d}=\zeta_{\theta}(t-s)$
is independent of $\zeta_{\theta}(s).$ We say that the process
$(\zeta_{\theta}(t); t\ge 0)$ is a GGC$(\theta,F_{X})$
subordinator. Proposition~\ref{prop1beta} shows how
Theorems~\ref{thm1beta} and~\ref{thm2beta}, can be used to address
the usually difficult problem of describing explicitly the
densities of the finite-dimensional distribution of a
subordinator~(see \cite{Kingman75}). This has implications in, for
instance, the explicit description of densities of Bayesian
nonparametric prior and posterior models. But clearly is of wider
interest in terms of the distribution theory of infinitely
divisible random variables and associated processes.

In Section~\ref{sec:gamma}, we describe how the operation of
exponentially titling the density of a GGC$(\theta, F_{X})$ random
variable leads to a relationship between the densities of the mean
functional $M_{\theta}(F_{X})$ and yet another family of mean
functionals. This is summarized in Theorem~\ref{thm1gamma}.
Section~\ref{sec:tiltbeta} then discusses a combination of the two
operations. Proposition~\ref{prop1gamma} describes the density of
beta scaled and tilted mean functionals of order 1. Using this,
Proposition~\ref{prop2gamma} describes a method to calculate a key
quantity in the explicit description of the density and cdf of
mean functionals. In section~\ref{sec:example} we demonstrate how
our results can be applied to extend and explain results related
to two well known cases of Dirichlet mean functionals. However, we
emphasize that Proposition~\ref{prop41},~\ref{prop42}
and~\ref{prop43} are genuinely new results to the literature. More
complex applications, which may be viewed as extensions of
section~\ref{sec:CifarelliMelilli}, may be found in an unpublished
preliminary version of this work in~\cite{JamesGamma}. We discuss
and develop these further in~James~\cite{JamesLinnik}. Section 5
presents a more involved result relative to those in section 4,
but which does not require a great deal of background material.
Here we show how the results in section 2 are used to derive the
finite dimensional distribution and related quantities of a class
of subordinators recently studied in~\cite{BFRY}.

\subsection{Preliminaries}\label{sec:prelim}
Suppose that $X$ is a positive random variable with
distribution $F_{X}$, and define the function
$$
\Phi_{F_{X}}(t)=\int_{0}^{\infty}\log(|t-x|)\indic(t\neq
x)F_{X}(dx)=\mathbb{E}[\log(|t-X|)\indic(t\neq X)].
$$
Furthermore, define
$$
\Delta_{\theta}(t|F_{X})=\frac{1}{\pi}\sin(\pi \theta
F_{X}(t)){\mbox e}^{-\theta \Phi_{F_{X}}(t)},
$$
where using a Lebesque-Stieltjes integral,
$F_{X}(t)=\int_{0}^{t}F_{X}(dx).$ Cifarelli and
Regazzini~\cite{CifarelliRegazzini}~(see
also~\cite{CifarelliMelilli}), apply inversion formula to
obtain the distributional formula for $M_{\theta}(F_{X})$
as follows. For all $\theta>0$, the cdf can be expressed as
\begin{equation}
\label{DPcdf} \int_{0}^{x}{(x-t)}^{\theta-1} \Delta_{\theta
}(t|F_{X})dt
\end{equation}
provided that $\theta F_{X}$ possesses no jumps of size
greater than or equal to one. If we let $\xi_{\theta
F_{X}}(\cdot)$ denote the density of $M_{\theta}(F_{X}),$
it takes its simplest form for $\theta=1$, which is
\begin{equation} \label{M1}
\xi_{F_{X}}(x)=\Delta_{1}(x|F_{X})=\frac{1}{\pi}\sin(\pi
F_{X}(x)){\mbox e}^{-\Phi(x)}. \end{equation} Density formulae for
$\theta>1$ are described as
\begin{equation} \label{M2}
\xi_{\theta F_{X}}(x)= {(\theta-1)}\int_{0}^{x}{(x-t)}^{\theta-2}
\Delta_{\theta }(t|F_{X})dt.
\end{equation}

An expression for the density, which holds for all
$\theta>0$, was recently obtained by James, Lijoi and
Pr\"unster~\cite{JLP} as follows,
\begin{equation} \label{generaldensity} \xi_{\theta F_{X}}(x)=
 \frac{1}{\pi}\int_{0}^{x}{(x-t)}^{\theta-1} d_{\theta}(t|F_{X})dt
\end{equation} where $$d_{\theta}(t|F_{X})=\frac{d}{dt}\sin(\pi \theta
F_{X}(t)){\mbox e}^{-\theta \Phi(t)}.$$ For additional
formula, see~\cite{CifarelliRegazzini, Regazzini2002,
Lijoi}.

\begin{rem} Throughout for random variables $R$ and $X,$ when we
write the product $RX$ we will assume unless otherwise
mentioned that $R$ and $X$ are independent. This convention
will also apply to the multiplication of the special random
variables that are expressed as mean functionals. That is
the product $M_{\theta}(F_{X})M_{\theta}(F_{Z})$ is
understood to be a product of independent Dirichlet means.
\end{rem}

\begin{rem} Throughout we will be using the fact that if $R$ is a gamma
random variable, then the independent random variables
$R,X,Z$ satisfying $RX\overset{d}=RZ$ imply that
$X\overset{d}=Z.$ This is true because gamma random
variables are simplifiable. For precise meaning of this
term and conditions, see Chaumont and Yor~\cite[sec. 1.12
and 1.13]{Chaumont}. This fact also applies to the case
where $R$ is a positive stable random variable.
\end{rem}

\section{Beta Scaling}~\label{sec:beta}In this section we investigate the simple operation of multiplying a Dirichlet mean functional
$M_{\theta}(F_{X})$ by certain beta random variables. Note first
that if $M$ denotes an arbitrary positive random variable with
density $f_{M},$ then by elementary arguments it follows that the
random variable $W\overset{d}=\beta_{a,b}M,$ where $\beta_{a,b}$
is beta$(a,b)$ independent of $M,$ has density expressible as
$$
f_{W}(w)=\frac{\Gamma(a+b)}{\Gamma(a)\Gamma(b)}\int_{0}^{1}f_{M}(w/u)u^{a-2}{(1-u)}^{b-1}du.
$$
However it is only in special cases where the density
$f_{W}$ can be expressed in even simpler terms. That is to
say, it is not obvious how to carry out the integration. In
the next results we show how remarkable simplifications can
be achieved when $M=M_{\theta}(F_{X}),$ in particular for
the range $0<\theta\leq 1,$ and $\beta_{a,b}$ is a
symmetric beta random variable. First we will need to
introduce some additional notation. Let $Y_{\sigma}$ denote
a Bernoulli random variable with success probability
$0<\sigma\leq 1.$ Then if $X$ is a random variable with
distribution $F_{X}$, independent of $Y_{\sigma},$ it
follows that the random variable $XY_{\sigma}$ has
distribution denoted as
\begin{equation}
F_{XY_{\sigma}}(dx)=\sigma
F_{X}(dx)+(1-\sigma)\delta_{0}(dx) \label{pdfp},
\end{equation}
and cdf
\begin{equation}
F_{XY_{\sigma}}(x)=\sigma F_{X}(x)+(1-\sigma)\indic(x\ge 0)
\label{cdfp}.
\end{equation}

Hence, there exists the mean functional
$$
M_{\theta}(F_{XY_{\sigma}})\overset{d}=\int_{0}^{\infty}y\tilde{P}_{\theta}(dy)
$$
where $\tilde{P}_{\theta}(dy)$ denotes a Dirichlet process
with parameters $(\theta, F_{XY_{\sigma}}).$ In addition we
have for $x>0,$
\begin{equation}
\Phi_{F_{XY_{\sigma}}}(x)=\mathbb{E}[\log(|x-XY_{\sigma}|)\indic(XY_{\sigma}\neq
x)]=\sigma\Phi_{F_{X}}(x)+(1-\sigma)\log(x) \label{Pid}.
\end{equation}
When $\sigma=1,$ $Y_{\sigma}=1$ and hence
$F_{XY_{1}}(x)=F_{X}(x).$ Let $E_{\sigma}$ denote a set
such that $\mathbb{E}[P_{\theta}(E_{\sigma})]=\sigma.$ Now
notice that every beta random variable, $\beta_{a,b},$
where $a, b$ are arbitrary positive constants, can be
represented as the simple mean functional,
$$
P_{\theta}(E_{\sigma})\overset{d}=\beta_{\theta
\sigma,\theta(1
-\sigma)}\overset{d}=M_{\theta}(F_{Y_{\sigma}}),$$ by
choosing
$$
\sigma=\frac{a}{a+b}{\mbox { and }}\theta=a+b.
$$
We note however that there are other choices of $F_{X}$
that will also yield beta random variables as mean
functionals. Throughout we will use the convention that
$\beta_{\theta,0}:=1,$ that is the case when $\sigma=1.$ We
now present our first result.

\begin{thm}~\label{thm1beta} For $\theta>0$ and $0<\sigma\leq 1$, let
$\beta_{\theta \sigma,\theta(1 -\sigma)}$ denote a beta
random variable with parameters $(\theta \sigma,
\theta(1-\sigma))$, independent of the mean functional
$M_{\theta \sigma}(F_{X}).$ Then
\begin{enumerate}
\item[(i)]$\beta_{\theta \sigma,\theta(1-\sigma)}M_{\theta
\sigma}(F_{X})\overset{d}=M_{\theta}(F_{XY_{\sigma}}).$
\item[(ii)]Equivalently, $M_{\theta}(F_{Y_{\sigma}})M_{\theta
\sigma}(F_{X})\overset{d}=M_{\theta}(F_{XY_{\sigma}}).$
\item[(iii)]$G_{\theta \sigma}M_{\theta \sigma}(F_{X})\overset{d}=G_{\theta}M_{\theta
}(F_{XY_{\sigma}})$
\item[(iv)]That is, GGC$(\theta \sigma,F_{X})=$GGC$(\theta,F_{XY_{\sigma}})$.
\end{enumerate}
\end{thm}
\begin{proof}Since $M_{\theta}(F_{Y_{\sigma}})\overset{d}=\beta_{\theta
\sigma,\theta(1-\sigma)}$ statements~(i) and~(ii) are
equivalent. We proceed by first establishing~(iii)
and~(iv). Note that using~(\ref{Levy}),
$$
\mathbb{E}[\log(1+\lambda XY_{\sigma})]=\sigma
\mathbb{E}[\log(1+\lambda
X)]=\sigma\int_{0}^{\infty}\log(1+\lambda x)F_{X}(dx).
$$
Hence
$$
\mathbb{E}[{\mbox e}^{-\lambda
G_{\theta}M_{\theta}(F_{XY_{\sigma}})}]={\mbox e}^{-\theta
\sigma \int_{0}^{\infty}\log(1+\lambda
x)F_{X}(dx)}=\mathbb{E}[{\mbox e}^{-\lambda G_{\theta\sigma
}M_{\theta\sigma}(F_{X})}],
$$
which means that
$G_{\theta}M_{\theta}(F_{XY_{\sigma}})\overset{d}=G_{\theta\sigma
}M_{\theta\sigma}(F_{X}),$ establishing statements~(iii)
and~(iv). Now writing
$G_{\theta\sigma}=G_{\theta}\beta_{\theta
\sigma,\theta(1-\sigma)}.$ It follows that
$$
G_{\theta}M_{\theta}(F_{XY_{\sigma}})\overset{d}=G_{\theta}\beta_{\theta
\sigma,\theta(1-\sigma)}M_{\theta\sigma}(F_{X}).$$ Hence $
\beta_{\theta
\sigma,\theta(1-\sigma)}M_{\theta\sigma}(F_{X})\overset{d}=M_{\theta}(F_{XY_{\sigma}}),
$ by the fact that gamma random variables are simplifiable.
\end{proof}

When $\theta=1$, we obtain results for random variables
$\beta_{\sigma,1-\sigma}M_{\sigma}(F_{X}).$ The symmetric
beta random variables $\beta_{\sigma,1-\sigma}$ arise in a
variety of important contexts, and are often referred to as
generalized arcsine laws with density expressible as
$$
\frac{\sin(\pi
\sigma)}{\pi}u^{\sigma-1}{(1-u)}^{-\sigma}{\mbox { for
}}0<u<1.
$$
Now using (\ref{pdfp}) and (\ref{cdfp}), let $
\mathcal{C}(F_{X})=\{x:F_{X}(x)>0\},$ then for $x>0,$
\begin{equation}
\label{sinp} \sin(\pi
F_{XY_{\sigma}}(x))=\left\{\begin{array}{cc}
  \sin(\pi\sigma[1-F_{X}(x)]), &   {\mbox{ if }} x\in\mathcal{C}(F_{X}),\\
 \sin(\pi(1-\sigma)),& {\mbox{ if }} x\notin\mathcal{C}(F_{X}).\\
                     \end{array}\right.
\end{equation}
Note also that $\sin(\pi[1-F_{X}(x)])=\sin(\pi F_{X}(x)).$
The next result yields another surprising property of these
random variables.

\begin{thm}~\label{thm2beta}Consider the setting in the Theorem~\ref{thm1beta}. Then when $\theta=1$, it follows that for each fixed $0<\sigma\leq 1,$ the random variable
$M_{1}(F_{XY_{\sigma}})\overset{d}=\beta_{\sigma,1-\sigma}M_{\sigma}(F_{X})$
has density
\begin{equation}
\xi_{F_{XY_{\sigma}}}(x)=\frac{x^{\sigma-1}}{\pi}\sin(\pi
F_{XY_{\sigma}}(x)){\mbox e}^{-\sigma\Phi_{F_{X}}(x)}{\mbox
{ for }} x>0, \label{betaid}
\end{equation}
specified by~(\ref{sinp}). Since
GGC$(\sigma,F_{X})=$GGC$(1,F_{XY_{\sigma}}),$ this implies
that the random variable
$G_{\sigma}M_{\sigma}(F_{X})\overset{d}=G_{1}M_{1}(F_{XY_{\sigma}})$
has density
\begin{equation}
g_{\sigma,F_{X}}(x)=\frac{1}{\pi}\int_{0}^{\infty}{\mbox
e}^{-\frac{x}{y}}y^{\sigma-2}\sin(\pi
F_{XY_{\sigma}}(y)){\mbox e}^{-\sigma\Phi_{F_{X}}(y)}dy
\label{mix}
\end{equation}
\end{thm}
\begin{proof}Since
$M_{1}(F_{XY_{\sigma}})\overset{d}=\beta_{\sigma,1-\sigma}M_{\sigma}(F_{X}),$
the density is of the form~(\ref{M1}),  for each fixed
$\sigma\in (0,1].$ Furthermore we use the identity
in~(\ref{Pid}).
\end{proof}
\begin{rem}
 It is worthwhile to mention that transforming to the random
 variable
$1/\beta_{\sigma,1-\sigma}$,~(\ref{betaid}) is equivalent
to the otherwise not obvious integral identity,
$$
\frac{\sin(\pi
\sigma)}{\pi}\int_{1}^{\infty}\frac{\xi_{\sigma
F_{X}}(xy)}{{(y-1)}^{\sigma}}dy
=\frac{x^{\sigma-1}}{\pi}\sin(\pi F_{XY_{\sigma}}(x)){\mbox
e}^{-\sigma\Phi(x)}.
$$
This leads to interesting results when the density
$\xi_{\sigma F_{X}}(x)$ has a known form. On the other
hand, we see that one does not need the explicit density of
$M_{\sigma}(F_{X})$ to obtain the density of
$M_{1}(F_{XY_{\sigma}})\overset{d}=\beta_{\sigma,1-\sigma}M_{\sigma}(F_{X}).$
In fact, owing to our goal of yielding simple densities for
many Dirichlet means from one mean, we see that the effort
to calculate the density of $M_{1}(F_{XY_{\sigma}}),$ for
each $0< \sigma\leq 1,$ is no more than what is needed to
calculate the density of $M_{1}(F_{X}).$
\end{rem}

We now see how this translates into the usually difficult
problem of describing explicitly the density of the
finite-dimensional distribution of a subordinator. In the
next result we use the notation $\zeta_{\theta}(C)$ to mean
$\int_{0}^{\infty}\indic\left(s\in
C\right)\zeta_{\theta}(ds).$

\begin{prop}~\label{prop1beta}Let $(\zeta_{\theta}(t);t\leq 1/\theta)$ denote a GGC$(\theta,F_{X})$ subordinator on $[0,1/\theta].$ Furthermore
let $(C_{1}, \ldots, C_{k})$ denote an arbitrary disjoint
partition of the interval $[0,1/\theta].$ Then the
finite-dimensional distribution
$(\zeta_{\theta}(C_{1}),\ldots,\zeta_{\theta}(C_{k}))$ has
a joint density
\begin{equation} \prod_{i=1}^{k}g_{\sigma_{i},F_{X}}(x_{i}),
\label{fidi1}
\end{equation}
where each $\sigma_{i}=\theta|C_{i}|>0$ and
$\sum_{i=1}^{k}\sigma_{i}=1.$ The density
$g_{\sigma_{i},F_{X}}$ is given by~(\ref{mix}). That is,
$\zeta_{\theta}(C_{i})\overset{d}=G_{1}M_{1}(F_{XY_{\sigma_{i}}})$
and are independent for $i=1,\ldots, k,$ where
$M_{1}(F_{XY_{\sigma_{i}}})\overset{d}=\beta_{\sigma_{i},1-\sigma_{i}}M_{\sigma_{i}}(F_{X})$
has density
$$
\frac{1}{\pi}x^{\sigma_{i}-1}\sin(\pi
F_{XY_{\sigma_{i}}}(x)){\mbox
e}^{-\sigma_{i}\Phi_{F_{X}}(x)}.
$$
\end{prop}

\begin{proof} First, since $(C_{1},\ldots,C_{k})$
partitions the interval $[0,1/\theta],$ it follows that
their sizes satisfy $0<|C_{k}|\leq 1/\theta$ and
$\sum_{i=1}^{k}|C_{k}|=1/\theta.$ Since $\zeta_{\theta}$ is
a subordinator the independence of the
$\zeta_{\theta}(C_{i})$ is a consequence of its independent
increment property. In fact these are essentially
equivalent statements. Hence, we can isolate each
$\zeta_{\theta}(C_{i}).$ It follows that for each $i$ the
Laplace transform is given by
$$
\mathbb{E}[{\mbox e}^{-\lambda
\zeta_{\theta}(C_{i})}]={\mbox e}^{-\theta
|C_{i}|\psi_{F_{X}}(\lambda)}={\mbox
e}^{-\sigma_{i}\psi_{F_{X}}(\lambda)},
$$
which shows that each $\zeta_{\theta}(C_{i})$ is
GGC$(\sigma_{i},F_{X})$ for $0<\sigma_{i}\leq 1.$ Hence the
result follows from Theorem~\ref{thm2beta}.
\end{proof}

\section{Exponential Tilting/Esscher Transform}\label{sec:gamma}
In this section we describe how the operation of
\emph{exponential tilting} of the density of a
GGC$(\theta,F_{X})$ random variable leads to a non-trivial
relationship between a mean functional determined by
$F_{X}$ and $\theta,$ and an entire family of mean
functionals indexed by an arbitrary constant $c>0.$
Additionally this will identify a non-obvious relationship
between two classes of mean functionals. Exponential
tilting is merely a catchy phrase for the operation of
applying an exponential change of measure to a density or
more general measure. In mathematical finance and other
applications it is known as an \emph{Esscher Transform}
which is a key tool for option pricing. We mention that
there is much known about exponential tilting of infinitely
divisible random variables and in fact
Bondesson~\cite[example 3.2.5]{BondBook} discusses
explicitly the case of GGC random variables, albeit not in
the way we shall describe it. In addition, examining the
gamma representation in~(\ref{gammarep}) one can see a
relationship to Lo and Weng~\cite[Proposition 3.1]{LW}~(see
also K\"uchler and Sorensen~\cite{Kuc} and
James~\cite[Proposition 2.1]{James05} for results on
exponential tilting of L\'evy processes). However, here our
focus is on the properties of related mean functionals
which leads to genuinely new insights.

Before we elaborate on this, we describe generically what
we mean by exponential tilting. Suppose that $T$ denotes an
arbitrary positive random variable with density, say
$f_{T}.$ It follows that for each positive $c,$ the random
variable $cT$ is well-defined and has density
$$
\frac{1}{c}f_{T}(t/c).
$$
Exponential tilting refers to the exponential change of
measure resulting in a random variable, say
$\tilde{T}_{c},$ defined by the density
$$
f_{\tilde{T}_{c}}(t)=\frac{{\mbox
e}^{-t}({1}/{c})f_{T}(t/c)}{\mathbb{E}[{\mbox e}^{-cT}]}.
$$
Thus from the random variable $T$ one gets a family of random
variables $(\tilde{T}_{c};c>0).$ Obviously the density for each
$\tilde{T_{c}}$ does not differ much. However something
interesting happens when $T$ is a scale mixture of a gamma random
variables, i.e., $T=G_{\theta}M,$ for some random positive random
variable $M$ independent of $G_{\theta}.$ In that case one can
show, see \cite{JamesGamma}, that $T_{c}=G_{\theta}\tilde{M}_{c}$
where $\tilde{M}_{c}$ is sufficiently distinct for each value of
$c.$ We demonstrate this for the case where $M=M_{\theta}(F_{X}).$

 First note that obviously,
 $cM_{\theta}(F_{X})=M_{\theta}(F_{cX}),$ for each $c>0,$ which in
 itself is not a very interesting transformation. Now setting
 $T_{\theta}=G_{\theta}M_{\theta}(F_{X})$ with density denoted as $g_{\theta,F_{X}}$, the corresponding
 random variable $\tilde{T}_{\theta,c}$
 resulting from exponential tilting has density
\begin{equation}
 {\mbox e}^{-t}(1/c)g_{\theta,F_{X}}(t/c){\mbox
 e}^{\theta\psi_{F_{X}}(c)}
 \label{tilt1}
 \end{equation}
and Laplace transform
\begin{equation}
\frac{\mathbb{E}[{\mbox
e}^{-c(1+\lambda)G_{\theta}M_{\theta}(F_{X})}]}{\mathbb{E}[{\mbox
e}^{-cG_{\theta}M_{\theta}(F_{X})}]}={\mbox
e}^{-\theta[\psi_{F_{X}}(c(1+\lambda))-\psi_{F_{X}}(c)]}.
\label{tilt2}
\end{equation}
Now for each $c>0,$ define the random variable
$$A_{c}\overset{d}=\frac{cX}{(cX+1)}.$$ That is, the cdf of the random
variable $A_{c},$ can be expressed as,
$$
F_{A_{c}}(y)=F_{X}\left(\frac{y}{c(1-y)}\right){\mbox { for
}}0<y<1.
$$
In the next theorem we will show that $M_{\theta}(F_{X})$ relates
to the family of mean functionals $(M_{\theta}(F_{A_{c}});c>0),$
by the tilting operation described above. Moreover, we will
describe the relationship between their densities.

\begin{thm}\label{thm1gamma} Suppose that $X$ has distribution $F_{X}$ and for each $c>0,$ $A_{c}\overset{d}=cX/(cX+1)$ is a random variable with distribution $F_{A_{c}}.$  For each $\theta>0$, let
$T_{\theta}=G_{\theta}M_{\theta}(F_{X})$ denote a GGC
$(\theta,F_{X})$ random variable having density
$g_{\theta,F_{X}}.$ Let $\tilde{T}_{\theta,c}$ denote a
random variable with density and Laplace transform
described by (\ref{tilt1}) and~(\ref{tilt2}) respectively.
Then $\tilde{T}_{\theta,c}$ is a GGC$(\theta,F_{A_{c}})$
random variable and hence representable as
$G_{\theta}M_{\theta}(F_{A_{c}}).$ Furthermore, the
following relationships exists between the densities of the
mean functionals $M_{\theta}(F_{X})$ and
$M_{\theta}(F_{A_{c}}).$
\begin{enumerate}
\item[(i)]Suppose that the density of $M_{\theta}(F_{X})$, say, $\xi_{\theta F_{X}}$
is known. Then the density of $M_{\theta}(F_{A_{c}})$ is
expressible as
$$
\xi_{\theta F_{A_{c}}}(y)=\frac{1}{c}{\mbox
e}^{\theta\psi_{F_{X}}(c)}{{(1-y)}^{\theta-2}}\xi_{\theta
F_{X}}\left(\frac{y}{c(1-y)}\right),$$ for $0<y<1.$
\item[(ii)]Conversely, if the density of $M_{\theta}(F_{A_{c}})$,
$\xi_{\theta F_{A_{c}}}(y),$ is known then the density of
$M_{\theta}(F_{X})$ is given by
$$
\xi_{\theta F_{X}}(x)={(1+x)}^{\theta-2}\xi_{\theta
F_{A_{1}}}\left(\frac{x}{1+x}\right) {\mbox
e}^{-\theta\psi_{F_{X}}(1)}.
$$
\end{enumerate}
\end{thm}
\begin{proof}We proceed by first examining the L\'evy exponent
$[\psi_{F_{X}}(c(1+\lambda))-\psi_{F_{X}}(c)]$ associated
with $\tilde{T}_{\theta,c}$ as described in~(\ref{tilt2}).
Notice that
$$
\psi_{F_{X}}(c(1+\lambda))=\int_{0}^{\infty}\log(1+c(1+\lambda)x)F_{X}(dx)
$$
and $\psi_{F_{X}}(c)$ is of the same form with $\lambda=0.$
Hence isolating the logarithmic terms we can focus on the
difference
$$
\log(1+c(1+\lambda)x)-\log(1+cx).
$$
This is equivalent to
$$
\log\left(1+\frac{cx}{1+cx}\lambda\right)=
\log\left(\frac{1}{1+cx}+\frac{cx}{1+cx}(1+\lambda)\right),
$$
showing that $\tilde{T}_{\theta,c}$ is
GGC$(\theta,F_{A_{c}}).$ This fact can also be deduced from
Proposition 3.1 in Lo and Weng~\cite{LW}. The next step is
to identify the density of $M_{\theta}(F_{A_c}),$ in terms
of the density of $M_{\theta}(F_{X}).$ Using the fact that
$T_{\theta}=G_{\theta}M_{\theta}(F_{X}),$ one may write the
density of $T_{\theta}$ in terms of a gamma mixture as
$$
g_{\theta,F_{X}}(t)=\frac{t^{\theta-1}}{\Gamma(\theta)}\int_{0}^{\infty}{\mbox
e}^{-t/m}m^{-\theta}\xi_{\theta F_{X}}(m)dm.
$$
Hence, rearranging terms in~(\ref{tilt1}), it follows that
the density of $\tilde{T}_{\theta,c}$ can be written as
$$
{\mbox e}^{\theta
\psi_{F_{X}}(c)}\frac{t^{\theta-1}}{\Gamma(\theta)}\int_{0}^{\infty}{\mbox
e}^{-t\frac{cm+1}{cm}}{(cm)}^{-\theta}\xi_{\theta
F_{X}}(m)dm.
$$
Now further algebraic manipulation makes this look like a
mixture of a gamma$(\theta,1)$ random variable, as follows,
$$
\frac{t^{\theta-1}}{\Gamma(\theta)}\int_{0}^{\infty}{\mbox
e}^{-t\frac{cm+1}{cm}}{\left[\frac{cm+1}{cm}\right]}^{\theta}\frac{{\mbox
e}^{\theta \psi_{F_{X}}(c)}\xi_{\theta
F_{X}}(m)}{{(1+cm)}^{\theta}}dm.
$$
Hence it is evident that $M_{\theta}(F_{A_{c}})$ has the
same distribution as a random variable $cM/(cM+1)$ where
$M$ has density
$$
{{\mbox e}^{\theta
\psi_{F_{X}}(c)}{(1+cm)}^{-\theta}\xi_{\theta F_{X}}(m)}.
$$
Thus statements~(i) and~(ii) follow.
\end{proof}

\subsection{Tilting and Beta Scaling}\label{sec:tiltbeta} This section describes what happens when one applies the exponentially tilting operation
relative to a mean functional resulting from beta scaling. Recall
that the tilting operation applied to
$G_{\theta}M_{\theta}(F_{X})$ described in the previous section
sets up a relationship between $M_{\theta}(F_{X})$ and
$M_{\theta}(F_{Ac}).$ Consider the random variable $\beta_{\theta
\sigma,\theta(1-\sigma)}M_{\theta
\sigma}(F_{X})\overset{d}=M_{\theta}(F_{XY_{\sigma}}).$ Then
tilting $G_{\theta}M_{\theta}(F_{XY_{\sigma}})$ as in the previous
section leads to the random variable
$G_{\theta}M_{\theta}(F_{cXY_{\sigma}/(cXY_{\sigma}+1)})$ and
hence relates $$\beta_{\theta \sigma,\theta(1-\sigma)}M_{\theta
\sigma}(F_{X})\overset{d}=M_{\theta}(F_{XY_{\sigma}})$$ to the
Dirichlet mean of order $\theta,$
$$
M_{\theta}(F_{cXY_{\sigma}/(cXY_{\sigma}+1)}).
$$
Now letting $F_{A_{c}Y_{\sigma}}$ denote the distribution
of $A_{c}Y_{\sigma},$ one has
$$
A_{c}Y_{\sigma}\overset{d}=\frac{cXY_{\sigma}}{(cXY_{\sigma}+1)}
$$
and hence
\begin{equation}
M_{\theta}(F_{cXY_{\sigma}/(cXY_{\sigma}+1)})\overset{d}=M_{\theta}
(F_{A_{c}Y_{\sigma}})\overset{d}=\beta_{\theta\sigma,\theta(1-\sigma)}M_{\theta\sigma}(F_{A_{c}}).
\end{equation}
In a way this shows that the order of beta scaling and
tilting can be interchanged. We now derive a result for the
cases of
$M_{1}(F_{XY_{\sigma}})=\beta_{\sigma,1-\sigma}M_{\sigma}(F_{X})$
and
$M_{1}(F_{A_{c}Y_{\sigma}})=\beta_{\sigma,1-\sigma}M_{\sigma}(F_{A_{c}}),$
related by the tilting operation described above. Combining
Theorem~\ref{thm2beta} with Theorem~\ref{thm1gamma} leads
to the following result.

\begin{prop}\label{prop1gamma} For each $0<\sigma\leq 1,$ the
random variables
$M_{1}(F_{XY_{\sigma}})=\beta_{\sigma,1-\sigma}M_{\sigma}(F_{X})$
and
$M_{1}(F_{A_{c}Y_{\sigma}})=\beta_{\sigma,1-\sigma}M_{\sigma}(F_{A_{c}})$
satisfy the following;
\begin{enumerate}
\item[(i)]The density of $M_{1}(F_{A_{c}Y_{\sigma}})$ is expressible as
$$
\xi_{F_{A_{c}Y_{\sigma}}}(y)=\frac{{\mbox
e}^{\sigma\psi_{F_{X}}(c)}y^{\sigma-1}}{\pi
c^{\sigma}{(1-y)}^{\sigma}}\sin\left[\pi
F_{XY_{\sigma}}\left(\frac{y}{c(1-y)}\right)\right]{\mbox
e}^{-\sigma\Phi_{F_{X}}\left(\frac{y}{c(1-y)}\right)}
$$
for $0<y<1.$
\item[(ii)]Conversely, the density of
$M_{1}(F_{XY_{\sigma}})$ is given by
$$
\xi_{F_{XY_{\sigma}}}(x)=\frac{{\mbox
e}^{-\sigma\psi_{F_{X}}(1)}
x^{\sigma-1}}{\pi{(1+x)}}\sin\left[\pi
F_{A_{1}Y_{\sigma}}\left(\frac{x}{1+x}\right)\right]{\mbox
e}^{-\sigma\Phi_{F_{A_{1}}}\left(\frac{x}{1+x}\right)}.
$$
\end{enumerate}
\end{prop}
\begin{proof}For clarity statement [(i)] is obtained by first using
Theorem~\ref{thm1gamma}. Which gives,
$$
\xi_{F_{A_{c}Y_{\sigma}}}(y)=\frac{1}{c}{\mbox
e}^{\psi_{F_{XY_{\sigma}}}(c)}{{(1-y)}^{-1}}\xi_{ F_{XY_
{\sigma}}}\left(\frac{y}{c(1-y)}\right),$$ for $0<y<1.$ It then
remains to substitute the form of the density~(\ref{betaid}) given
in Theorem~\ref{thm2beta}. Statement [(ii)] proceeds in the same
way using~(\ref{mix}).
\end{proof}

Note that even if one can calculate $\Phi_{F_{A_{c}}}$ for
some fixed value of $c,$ it may not be so obvious how to
calculate it for another value. The previous results allow
us to relate their calculation to that of $\Phi_{F_{X}}$ as
described next.

\begin{prop}\label{prop2gamma} Set $A_{c}=cX/(cX+1)$ and define $\Phi_{F_{A_{c}}}(y)=\mathbb{E}[\log(|y-A_{c}|)\indic(A_{c}\neq
y)].$ Then for $0<y<1,$
$$
\Phi_{F_{A_{c}}}(y)=\Phi_{F_{X}}\left(\frac{y}{c(1-y)}\right)
-\psi_{F_{X}}(c)+\log(c(1-y)).
$$
\end{prop}

\begin{proof}The result can be deduced by using Proposition~\ref{prop1gamma} in
the case of $\sigma=1.$ First notice that $\sin(\pi
F_{X}(\frac{y}{c(1-y)}))=\sin(\pi F_{A_{c}}(y)).$ Now
equating the form of the density of $M_{1}(F_{A_{c}})$
given by~(\ref{M1}) with the expression given in
Proposition~\ref{prop1gamma}. It follows that
$$
{\mbox e}^{-\Phi_{F_{A_{c}}}(y)}=\frac{{\mbox
e}^{\psi_{F_{X}}(c)}}{ c{(1-y)}}{\mbox
e}^{-\Phi_{F_{X}}\left(\frac{y}{c(1-y)}\right)},
$$
which yields the result.
\end{proof}

\begin{rem} We point out that if $G_{\kappa}$ represents a gamma
random variable for $\kappa\neq \theta$, independent of
$M_{\theta}(F_{X}),$ it is not necessarily true that
$G_{\kappa}M_{\theta}(F_{X})$ is a GGC random variable. For
this to be true $M_{\theta}(F_{X})$ would need to be
equivalent in distribution to some $M_{\kappa}(F_{R}).$ In
that case, our results above would be applied for a
GGC$(\kappa, F_{R})$ model. We will encounter a variation
of such a situation in section~\ref{sec:Diaconis}.
\end{rem}

\section{First Examples}\label{sec:example}
In this section we will demonstrate how our results in
section~\ref{sec:beta} and~\ref{sec:gamma} can be applied
to extend and explain results related to two well known
cases of Dirichlet mean functionals.

\subsection{An example connected to Diaconis and Kemperman}\label{sec:Diaconis}
Set $X=U,$ denoting a uniform$[0,1]$ random variable with
$F_{U}(u)=u$ for $0\leq u\leq 1.$ It is known from Diaconis and
Kemperman~\cite{Diaconis} that the density of $M_{1}(F_{U})$ is
\begin{equation} \label{DPUni} \frac{e}{\pi}\sin(\pi
y)y^{-y}{(1-y)}^{-(1-y)}{\mbox { for }}0<y<1.
\end{equation} Note furthermore that
$\tilde{T}_{1}\overset{d}=G_{1}M_{1}(F_{U})$ is
GGC$(1,F_{U})$ and has a rather strange Laplace transform,
$$
\mathbb{E}[{\mbox e}^{-\lambda G_{1}M_{1}(F_{U})}]={\mbox
e}^{-\psi_{F_{U}}(\lambda)}={\mbox
e}{(1+\lambda)}^{-\frac{\lambda+1}{\lambda}}.
$$
We can use this fact combined with the previous results to
obtain a new explicit expression for the density of what we
believe should be an important mean functional and
corresponding infinitely divisible random variable.

\begin{prop}\label{prop41} Let $W=G_{1}/G'_{1}$ be the ratio of two independent exponential $(1)$ random variables having density
$F_{W}(dx)/dx=(1+x)^{-2}{\mbox { for }}x>0$. Now let
$T_{1}$ denote a GGC$(1,F_{W})$ random variable, with
$\psi_{F_{W}}(\lambda)=\frac{\lambda}{\lambda-1}\log(\lambda).$
Then $T_{1}\overset{d}=G_{1}M_{1}(F_{W})$, where
$M_{1}(F_{W})$ has density,
$$
\xi_{F_{W}}(x)=\frac{1}{\pi}\sin\left(\frac{\pi
x}{1+x}\right)x^{-\frac{x}{1+x}}{\mbox { for }}x>0.
$$
\end{prop}

\begin{proof} First note that it is straightforward to show that
$$\mathbb{E}[{\mbox
e}^{-T_{1}}]=\mathbb{E}[{(1+M_{1}(F_{W}))}^{-1}]=e^{-1}.$$
This fact also establishes the existence of $T_{1}.$ Now we
see that
$$\frac{W}{W+1}=\frac{G_{1}}{(G_{1}+G'_{1})}\overset{d}=U.$$
Hence Theorem~\ref{thm1gamma} shows that $M_{1}(F_{U})$
arises from tilting the density of $G_{1}M_{1}(F_{W}).$ The
density is obtained by applying statement~(ii) of
Theorem~\ref{thm1gamma} to~(\ref{DPUni}). Or by
statement~(ii) of Proposition~\ref{prop1gamma}.\end{proof}

We now apply Theorem~\ref{thm2beta} and
Proposition~\ref{prop1beta} to give a description of the
finite-dimensional distribution of the subordinators
associated with the two random variables above. Here $U$
and $W=G_{1}/G'_{1}$ are as described previously.

\begin{prop}\label{prop42}Let $U$ denote a uniform$[0,1]$ random variable and let $W=G_{1}/G'_{1}$
denote a ratio of independent exponential$(1)$ random
variables. \noindent \begin{enumerate}\item[(i)] Suppose
that $(\tilde{\zeta}_{1}(t); 0<t<1)$ is a GGC$(1,F_{U})$
subordinator, then for $(C_{1},\ldots,C_{k})$ a disjoint
partition of $(0,1),$ the finite-dimensional distribution
has a joint density as in~(\ref{fidi1}), with,
$$
g_{\sigma_{i},F_{U}}(x_{i})=\int_{0}^{1}{\mbox
e}^{-\frac{x_{i}}{y}}\frac{e^{\sigma_{i}}}{\pi}\sin(\pi
\sigma_{i}(1-
y))y^{\sigma_{i}(1-y)-3}{(1-y)}^{-\sigma_{i}(1-y)}dy,
$$
for $i=1,\ldots,k.$
\item[(ii)] That is
$\tilde{\zeta}_{1}(C_{i})\overset{d}=G_{1}M_{1}(F_{UY_{\sigma_{i}}})$
and are independent for $i=1,\ldots, k.$ Furthermore, the
density of
$M_{1}(F_{UY_{\sigma_{i}}})\overset{d}=\beta_{\sigma_{i},1-\sigma_{i}}M_{\sigma_{i}}(F_{U})$
is
$$
\frac{e^{\sigma_{i}}}{\pi}\sin(\pi \sigma_{i}(1-
y))y^{\sigma_{i}(1-y)-1}{(1-y)}^{-\sigma_{i}(1-y)}
$$
for $0<y<1.$
\item[(iii)]If $(\zeta_{1}(t); 0<t<1)$ is a GGC$(1,F_{W})$ subordinator then
the finite-dimensional distribution
$(\zeta_{1}(C_{1}),\ldots,\zeta_{1}(C_{k}))$ is described
now with
$$
g_{\sigma_{i},F_{W}}(x_{i})=\int_{0}^{\infty}{\mbox
e}^{-\frac{x_{i}}{w}}\frac{1}{\pi}\sin\left(\frac{\pi
\sigma_{i}}{1+w}\right)w^{\frac{\sigma_{i}}{1+w}-3}dw.
$$
\item[(iv)] That is
$\zeta_{1}(C_{i})\overset{d}=G_{1}M_{1}(F_{WY_{\sigma_{i}}})$
and are independent for $i=1,\ldots, k.$ Furthermore, the
density of
$M_{1}(F_{WY_{\sigma_{i}}})\overset{d}=\beta_{\sigma_{i},1-\sigma_{i}}M_{\sigma_{i}}(F_{W})$
is
$$
\frac{1}{\pi}\sin\left(\frac{\pi
\sigma_{i}}{1+x}\right)x^{\frac{\sigma_{i}}{1+x}-1}
$$
for $x>0.$
\end{enumerate}
\end{prop}

\begin{proof} This now follows from Theorem~\ref{thm2beta}, Proposition~\ref{prop1beta} and~(\ref{DPUni}). Specifically, note that
$\mathcal{C}(F_{U})=(0,\infty),$ hence for any
$0<\sigma<1,$
$$
\sin(\pi F_{UY_{\sigma}}(u))=\sin(\pi \sigma(1-u))
$$
for $0<u<1$ and $0$ otherwise. Furthermore
from~(\ref{DPUni}), or by direct argument, it is easy to
see that,
$$
\Phi_{F_{U}}(y)=-\log\left(y^{-y}{(1-y)}^{-(1-y)}\right)-1.
$$
This fact also is evident from Diaconis and
Kemperman~\cite{Diaconis}. It follows that
$M_{1}(F_{UY_{\sigma}})$ has density
$$
\frac{e^{\sigma}}{\pi}\sin(\pi \sigma(1-
y))y^{\sigma(1-y)-1}{(1-y)}^{-\sigma(1-y)}{\mbox { for
}}0<y<1.
$$ The density for $M_{1}(F_{WY_{\sigma}})$ is obtained in a similar fashion by Proposition~\ref{prop1gamma}.
\end{proof}

\begin{rem}Setting
$$
A_{c}\overset{d}=\frac{cG_{1}}{cG_{1}+G'_{1}},
$$
one can easily obtain the density of the random variable
$M_{1}(F_{A_{c}})$ for each $c>0$ by using statement~(ii)
of Theorem~\ref{prop1gamma}. Note also that one can deduce
from the density of $M_{1}(F_{W})$ that
$\Phi_{F_{W}}(x)=[x/(1+x)]\log(x).$ Hence in this case an
application of Proposition~\ref{prop2gamma} shows that,
$$
\Phi_{F_{A_{c}}}(y)=
\frac{y}{c(1-y)+y}\log\left[\frac{y}{c(1-y)}\right]-\frac{c\log(c)}{c-1}+\log[c(1-y)].
$$
We note that otherwise it is not easy to calculate
$\Phi_{A_{c}}$, in this case, by direct arguments.
\end{rem}

\subsection{Reconciling some results of Cifarelli and Mellili}\label{sec:CifarelliMelilli} To
further illustrate our point we show how to reconcile two
apparently unrelated results given in Cifarelli and
Melilli~\cite{CifarelliMelilli}. Let $\beta_{1/2,1/2}$
denote a beta(1/2,1/2) random variable with cdf
$$F_{\beta_{1/2,1/2}}(y)=\frac{2}{\pi}\tan^{-1}(\sqrt{\frac{y}{1-y}})=1-\frac{2}{\pi}\cot^{-1}(\sqrt{\frac{y}{1-y}})
$$
often referred to as the arcsine law. Cifarelli and
Melilli~\cite[p.1394-1395]{CifarelliMelilli} show that for all
$\theta>0,$
$M_{\theta}(F_{\beta_{1/2,1/2}})\overset{d}=\beta_{\theta+1/2,\theta+1/2}.$
Now define the probability density
$$
\varrho_{1/2}(dx)/dx=\frac{1}{\pi}x^{-1/2}{(1+x)}^{-1}.
$$
Cifarelli and Mellili~\cite[p.1394-1395]{CifarelliMelilli} also
show that for $\theta\ge 1,$ $M_{\theta}(\varrho_{1/2})$ has the
density,
\begin{equation}
\frac{\Gamma(\theta+1)}{\Gamma(\theta+\frac{1}{2})\Gamma(\frac{1}{2})}x^{\theta-1/2}{(1+x)}^{-(\theta+1)}.
\label{bg}
\end{equation}
Hjort and Ongaro~\cite{Hjort} recently extend this result
for all $\theta>0.$

Here, however, we note that if $X$ has distribution
$\varrho_{1/2}$ then $X\overset{d}=G_{1/2}/G'_{1/2}$ where
$G_{1/2}$ and $G'_{1/2}$ are independent and identically
distributed gamma random variables. That is
$M_{\theta}(\varrho_{1/2})=M_{\theta}(F_{G_{1/2}/G'_{1/2}}).$
Now using the known fact that
$$
\beta_{1/2,1/2}\overset{d}=\frac{G_{1/2}/G'_{1/2}}{G_{1/2}/G'_{1/2}+1},
$$
we see that $\beta_{1/2,1/2}$ is a special case of $A_{1}$
in Proposition~\ref{prop2gamma}. That is,
$M_{\theta}(F_{\beta_{1/2,1/2}})$ results from
exponentially tilting the density of
$G_{\theta}M_{\theta}(F_{G_{1/2}/G'_{1/2}}),$ with $c=1.$
Hence one could have obtained the density of
$M_{\theta}(F_{G_{1/2}/G'_{1/2}})$ by using the result of
Cifarelli and Melili that
$M_{\theta}(F_{\beta_{1/2,1/2}})\overset{d}=\beta_{\theta+1/2,\theta+1/2},$and
simply applying statement~(ii) of Theorem~\ref{prop1gamma}.

\begin{rem}Note that the following identity for the normalizing constant,
$$
\frac{\Gamma(2\theta+1)}{\Gamma(\theta+1/2)\Gamma(\theta+1/2)}4^{-\theta}=
\frac{\Gamma(\theta+1)}{\Gamma(\theta+1/2)\Gamma(1/2)},
$$
is obtained from the gamma duplication formula, where in
particular
$$
\Gamma(2\theta+1)=\frac{4^{\theta}}{\Gamma(1/2)}{\Gamma(\theta+1)\Gamma(\theta+1/2)}.
$$
\end{rem}

Hereafter we shall use the notation
$X_{1/2}\overset{d}=G_{1/2}/G'_{1/2}.$ Moreover, from the
density in (\ref{bg}), it follows that, using the notation
$X_{1/2,\theta},$
$$
X_{1/2,\theta}\overset{d}=M_{\theta}(F_{X_{1/2}})\overset{d}=\frac{G_{\theta+1/2}}{G_{1/2}}.
$$
In particular $X_{1/2,0}=X_{1/2}.$ Furthermore,
$$
G_{\theta}X_{1/2,\theta}\overset{d}=G_{\theta}M_{\theta}(F_{X_{1/2}})\overset{d}=G_{\theta}\frac{G_{\theta+1/2}}{G_{1/2}}
$$
is GGC$(\theta,F_{X_{1/2}}),$ with Laplace transform
\begin{equation}
\mathbb{E}[{\mbox e}^{-\lambda
G_{\theta}M_{\theta}(F_{X_{1/2}})}]={(1+\sqrt{\lambda})}^{-2\theta}.
\label{linnik}
\end{equation}
Now, for $0<p<1,$ setting $c=p^{2}/{(1-p)}^{2}$ one can
extend the result of Cifarelli and Melilli for mean
functionals defined by the random variable,
$$
B_{1/2,p}=\frac{p^{2}G_{1/2}}{[p^{2}G_{1/2}+q^{2}G'_{1/2}]}=\frac{cX_{1/2}}{cX_{1/2}+1}
$$
having density
$$
\frac{qp}{\pi}\frac{y^{-1/2}{(1-y)}^{-1/2}}{q^{2}y+p^{2}(1-y)}{\mbox
{ for }}0<y<1,
$$
as follows, first the density of the random variable
$$cM_{\theta}(F_{X_{1/2}})=M_{\theta}(F_{cX_{1/2}})$$
is simply,
$$
\xi_{\theta F_{cX_{{1}/{2}}
}}(x)=\frac{\Gamma(\theta+1)}{\Gamma(\theta+\frac{1}{2})\Gamma(\frac{1}{2})}c^{1/2}x^{\theta-1/2}{(c+x)}^{-(\theta+1)}.
$$
The density of $M_{\theta}(F_{B_{{1}/{2},p}})$ is then
obtained by applying statement~(ii) of
Theorem~\ref{prop1gamma} to the density above, to get,
\begin{equation} \xi_{\theta
F_{B_{{1}/{2},p}}}(y)=pq\frac{\Gamma(\theta+1)}{\Gamma(\theta+\frac{1}{2})\Gamma(\frac{1}{2})}\frac{y^{\theta-1/2}{(1-y)}^{\theta-1/2}}{p^{2}(1-y)+q^{2}y}.
\label{Carlton}
\end{equation}
However this density equates with the density of a two parameter
$(1/2,\theta)$ Poisson-Dirichlet random probability measure
evaluated at a sect $C$ say, $\tilde{P}_{1/2,\theta}(C)$, such
that $\mathbb{E}[\tilde{P}_{1/2,\theta}(C)]=p,$ given in
Carlton~\cite[Remark 3.1]{Carlton}~(see also~\cite{JLP}
and~\cite{PY97}). Hence we see that
$M_{\theta}(F_{B_{{1}/{2},p}})\overset{d}=\tilde{P}_{1/2,\theta}(C)$
arises from exponentially tilting the density of
$G_{\theta}M_{\theta}(F_{X_{1/2}}).$

\subsubsection{Results for
$\theta>-1/2$} Notice that the random variable
$$
X_{1/2,\theta}\overset{d}=\frac{G_{\theta+1/2}}{G_{1/2}}
$$
is well defined for all $\theta>-1/2,$ and still has
density~(\ref{bg}) for the negative range of $\theta.$ However
obviously in the range $-1/2<\theta\leq 0$ it is not a Dirichlet
mean functional of order $\theta,$ and as such is not covered by
the result in~\cite{CifarelliMelilli}. The next result shows that
for all $\theta>-1/2$ the random variable $X_{1/2,\theta}$ is a
mean functional of order $1+\theta.$

\begin{prop}\label{prop43} For each $\theta>-1/2$, the random variable $X_{1/2,\theta}$ satisfies the distributional equality.
$$
X_{1/2,\theta}\overset{d}=\beta_{\theta+1/2,1/2}X_{1/2,\theta+1/2}\overset{d}=M_{1+\theta}(F_{X_{1/2}Y_{\frac{\theta+1/2}{1+\theta}}})
$$
As a consequence the random variable,
$G_{1+\theta}X_{1/2,\theta}\overset{d}=G_{\theta+1/2}X_{1/2,\theta+1/2}$
is
GGC$(\theta+1,F_{X_{1/2}Y_{\frac{\theta+1/2}{1+\theta}}})=$GGC$(\theta+1/2,F_{X}).$
\end{prop}

\begin{proof}To see this note the simple identity
$$
G_{\theta+1/2}\overset{d}=\beta_{\theta+1/2,1/2}G_{\theta+1}.
$$
The first equality in the Proposition is then obtained by dividing
both sides by an independent $G_{1/2}$ random variable. The second
equality is obtained by using Theorem~\ref{prop1beta}. That is
using the fact that
$\beta_{\theta+1/2,1/2}X_{1/2,\theta+1/2}\overset{d}=\beta_{\theta+1/2,1/2}M_{\theta+1/2}(F_{X_{1/2}}).$
\end{proof}

Thus we see that for $\theta>0,$
$$
X_{1/2,\theta}\overset{d}=M_{\theta}(F_{X_{1/2}})\overset{d}=M_{1+\theta}(F_{X_{1/2}Y_{\frac{\theta+1/2}{1+\theta}}}).
$$

\begin{rem}The random variable $X_{1/2}=G_{1/2}/G'_{1/2}$ is a
special case of an important random variable studied by
Lamperti~\cite{Lamperti}. The general case is defined for
$0<\alpha<1$, $X_{\alpha}=S_{\alpha}/S'_{\alpha}$, where
$S_{\alpha}$ and $S'_{\alpha}$ are iid positive stable
random variables of index $\alpha$. The GGC$(\theta,
F_{X_{1/2}})$ random defined by its Laplace
transform~(\ref{linnik}) is an example of the generalized
positive Linnik random variable which can be represented as
${(G_{\theta/\alpha})}^{1/\alpha}S_{\alpha}.$ Where in this
case $\alpha=1/2$. A thorough investigation of this class,
 and its various implications, initiated in~\cite{JamesGamma}, will
be discussed in a forthcoming
manuscript~\cite{JamesLinnik}~(see also~\cite{JamesYor}).
\end{rem}

\section{Obtaining the finite dimensional distribution of a subordinator of BFRY}
The previous examples revisited some cases that have existed in
the literature for some time. Our final example shows how one can
apply the results in section 2 to  obtain new results for a
subordinator recently studied by Bertoin, Fujita, Roynette and
Yor~\cite{BFRY}. We first supply some background.  For
$0<\alpha<1,$ let $S_{\alpha}$ denote a positive $\alpha$-stable
random variable specified by its Laplace transform
$$
\mathbb{E}[{\mbox e}^{-\lambda S_{\alpha}}]={\mbox
e}^{-\lambda^{\alpha}}.
$$
In addition define
$$
Z_{\alpha}={(\frac{S_{\alpha}}{S'_{\alpha}})}^{\alpha}
$$
where $S'_{\alpha}$ is independent of $S_{\alpha}$ and has the
same distribution. The density of this random variable was
obtained by  Lamperti\cite{Lamperti}(see also Chaumont and
Yor~(\cite{Chaumont}, exercise 4.2.1) and has the remarkably
simple form,
$$
f_{Z_{\alpha}}(y)=\frac{\sin(\pi \alpha)}{\pi
\alpha}\frac{1}{y^{2}+2y\cos(\pi \alpha)+1}{\mbox { for }}y>0.
$$

In order to avoid confusion we will now denote relevant random
variables appearing originally in~\cite{BFRY} as $\Delta_{\alpha}$
and $G_{\alpha}$, as $\Sigma_{\alpha}$ and $\mathbb{G}_{\alpha}$
respectively. From~\cite{BFRY}, let $(\Sigma_{\alpha}(t);t>0)$
denote a subordinator such that
\begin{eqnarray*}
\mathbb{E}[{\mbox e}^{-\lambda
\Sigma_{\alpha}(t)}]&=&{\big((\lambda+1)^{\alpha}-\lambda^{\alpha}\big)}^{t}\\
&=&\exp(-t(1-\alpha)\mathbb{E}[\log(1+1/\mathbb{G}_{\alpha})])
\end{eqnarray*}
where from (\cite{BFRY},  Theorems 1.1 and 1.3),
$\mathbb{G}_{\alpha}$ denotes a random variable such that
$$
\mathbb{G}_{\alpha}\overset{d}=\frac{Z^{1/\alpha}_{1-\alpha}}{1+Z^{1/\alpha}_{1-\alpha}}
$$
and has density on $(0,1)$ given by
$$
f_{\mathbb{G}_{\alpha}}(u)=\frac{\alpha
\sin(\pi\alpha)}{(1-\alpha)\pi}\frac{u^{\alpha-1}{(1-u)}^{\alpha-1}}{u^{2\alpha}-2(1-u)^{\alpha}\cos(\pi\alpha)+{(1-u)}^{2\alpha}}
$$
Hence it follows that the random variable $1/\mathbb{G}_{\alpha}$
takes its values on $(1,\infty)$ with probability one and has cdf
satisfying,
$$
1-F_{1/\mathbb{G}_{\alpha}}(x)=F_{\mathbb{G}_{\alpha}}(1/x)=F_{Z_{1-\alpha}}({(x-1)}^{-\alpha}),
$$
where using properties of the cdf $Z_{1-\alpha},$ (see \cite{FY}
and (\cite{JamesLinnik}, Proposition 2.1)), it follows that for
$x>1$

$$
F_{Z_{1-\alpha}}(\frac{1}{{(x-1)}^{\alpha}})=1-\frac{1}{\pi(1-\alpha)}\cot^{-1}\left(\frac{\cos(\pi
(1-\alpha))+(x-1)^{-\alpha}}{\sin(\pi (1-\alpha))}\right).
$$

It follows as noted by \cite{BFRY} that $(\Sigma_{\alpha}(t);t>0)$
is a GGC$(1-\alpha, F_{1/\mathbb{G}_{\alpha}})$ subordinator.
Where from (\cite{BFRY}, Theorem 1.1) the
GGC$(1-\alpha,F_{1/\mathbb{G}_{\alpha}})$ random variable
$\Sigma_{\alpha}\overset{d}=\Sigma_{\alpha}(1)$ satisfies,
$$
\Sigma_{\alpha}\overset{d}=\frac{G_{1-\alpha}}{\beta_{\alpha,1}}\overset{d}=\frac{G_{1-\alpha}}{U^{1/\alpha}}
$$
where $U$ denotes a uniform$[0,1]$ random variable and for clarity
$G_{1-\alpha}$ is a gamma$(1-\alpha,1)$ random variable. It is
evident, as investigated in \cite{FY}, that
$$
M_{1-\alpha}(F_{1/\mathbb{G}_{\alpha}})\overset{d}=\frac{1}{\beta_{\alpha,1}}\overset{d}=U^{-1/\alpha}
$$
So this constitutes a rare known example where a Dirichlet mean of
order $0<\theta<1$ has a simple recognizable form. Note also that
$M_{1-\alpha}(F_{\mathbb{G}_{\alpha}})\overset{d}=U$ for all
$0<\alpha\leq 1.$ The above points may also be found in the survey
paper of~\cite{JRY}.

Note that when $\alpha=1/2,$
$\mathbb{G}_{1/2}\overset{d}=\beta_{1/2,1/2}$, which then relates
to the results in section 4.2. In addition, it is known that for
each fixed $t,$
$$
\Sigma_{1/2}(t)\overset{d}=\frac{G_{t/2}}{\beta_{1/2,(1+t)/2}}.
$$
This result may be found in \cite{JamesYor} where the case of
$\alpha=1/2$ also belongs to another class of subordinators
indexed by $\alpha$.

Now using the fact discussed in this section we will show how to
use the results in section 2 to explicitly describe the finite
dimensional distribution of the subordinator $\Sigma_{\alpha}(t)$
over the range $0<t\leq 1/(1-\alpha)$, hence by infinite
divisibility for all $t.$ Additionally the analysis will also
yield expressions for mean functionals based on
$F_{1/\mathbb{G}_{\alpha}}.$ We first show how to calculate
\begin{equation}
\mathcal{R}_{\alpha}(x)=\Phi_{F_{1/\mathbb{G}_{\alpha}}}(x)=\mathbb{E}[\log(|t-1/\mathbb{G}_{\alpha}|)\indic(t\neq
1/\mathbb{G}_{\alpha})]. \label{phifunction}
\end{equation}
Using simple beta gamma algebra one has
$$
\Sigma_{\alpha}\overset{d}=\frac{G_{1-\alpha}}{\beta_{\alpha,1}}\overset{d}=G_{1}\frac{\beta_{1-\alpha,\alpha}}{U^{1/\alpha}}
$$
Hence applying Theorem 2.1, with $\theta=1,$ and
$\sigma=1-\alpha,$ it follows that $\Sigma_{\alpha}$ is also
GGC$(1,F_{Y_{1-\alpha}/\mathbb{G}_{\alpha}})$ and \begin{equation}
\frac{\beta_{1-\alpha,\alpha}}{\beta_{\alpha,1}}\overset{d}=\frac{\beta_{1-\alpha,\alpha}}{U^{1/\alpha}}\overset{d}=M_{1}(F_{Y_{1-\alpha}/\mathbb{G}_{\alpha}})
\label{key5}
\end{equation}
 Furthermore, using~(\ref{sinp}), one has
\begin{equation}
\label{sinap} \sin(\pi
F_{Y_{1-\alpha}/\mathbb{G}_{\alpha}}(x))=\left\{\begin{array}{cc}
  \sin(\pi(1-\alpha) F_{\mathbb{G}_{\alpha}}(1/x)), &   {\mbox{ if }} x> 1,\\
 \sin(\pi (1-\alpha)),& {\mbox{ if }} 0<x\leq 1.\\
                     \end{array}\right.
\end{equation}
where again using the properties of $F_{Z_{1-\alpha}},$ as deduced
from (\cite{JamesLinnik}, Proposition 2.1, [(iii)]),
\begin{equation}
\label{cdfid}
 \sin(\pi(1-\alpha)
F_{\mathbb{G}_{\alpha}}(1/x))=\frac{\sin(\pi(1-\alpha))}{{[{(x-1)}^{2\alpha}-2{(x-1)}^{\alpha}\cos(\pi
\alpha)+1]}^{1/2}}
\end{equation}
for $x>1.$ These points lead to the following description of
~(\ref{phifunction}).
\begin{prop}For $0<\alpha<1,$ consider
$\mathcal{R}_{\alpha}(x)$ as defined in (\ref{phifunction}). Then,
\begin{equation}
\label{rfunction} \mathcal{R}_{\alpha}(x)=\left\{\begin{array}{ll}
  \frac{1}{2(1-\alpha)}[\log(\frac{x^{2}}{[(x-1)^{2\alpha}-2(x-1)^{\alpha}\cos(\pi\alpha)+1]})], &   {\mbox{ if }} x> 1,\\
 \frac{1}{1-\alpha}\log({x}/{[1-{(1-x)}^{\alpha}]}),& {\mbox{ if }} 0<x\leq 1.\\
                     \end{array}\right.
\end{equation}
\end{prop}
\begin{proof}By standard calculations the density of
$B_{\alpha}=\beta_{1-\alpha,\alpha}/\beta_{\alpha,1}$ is given by
$$
f_{B_{\alpha}}(x)=\frac{\sin(\pi
(1-\alpha))}{\pi}x^{-\alpha-1}[1-{(1-x)}^{\alpha}\indic(x\leq 1)]
$$
However we see from~(\ref{key5}) that $B_{\alpha}\overset{d}=
M_{1}(F_{Y_{1-\alpha}/\mathbb{G}_{\alpha}}).$ Hence Theorem 2.2
applies and the density of $B_{\alpha}$ can be written as
$$
f_{B_{\alpha}}(x)=\frac{x^{-\alpha}}{\pi}\sin(\pi
F_{Y_{1-\alpha}/\mathbb{G}_{\alpha}}(x)){\mbox
e}^{-(1-\alpha)\mathcal{R}_{\alpha}(x)}
$$
Now equating the two forms of the density of $B_{\alpha}$ and
using (\ref{sinap}) and~(\ref{cdfid}), one then obtains the
expression for $\mathcal{R}_{\alpha}$
\end{proof}
This leads to our next result,
\begin{prop}For $0<\sigma\leq 1$ the density of
$M_{1}(F_{Y_{\sigma}/\mathbb{G}_{\alpha}})$ is given by,
$$
f_{\alpha,\sigma}(x)=\frac{x^{-\frac{\sigma\alpha}{1-\alpha}-1}}{\pi}\sin(\pi\sigma){[1-{(1-x)}^{\alpha}]}^{\frac{\sigma}{1-\alpha}}
$$
for $0<x\leq 1,$ and for $x>1$
$$
f_{\alpha,\sigma}(x)=\frac{x^{-\frac{\sigma\alpha}{1-\alpha}-1}}{\pi}\sin(\pi\sigma
F_{\mathbb{G}_{\alpha}}(1/x)){[(x-1)^{2\alpha}-2(x-1)^{\alpha}\cos(\pi\alpha)+1]}^{\frac{\sigma}{2(1-\alpha)}}.
$$
\end{prop}
\begin{proof} From Theorem 2.2 we have the general form of the density of $M_{1}(F_{Y_{\sigma}/\mathbb{G}_{\alpha}})$ is given by
$$
\frac{x^{\sigma-1}}{\pi}\sin(\pi
F_{Y_{\sigma}/\mathbb{G}_{\alpha}}(x)){\mbox e}^{-\sigma
\mathcal{R}_{\alpha}(x)}
$$
The result is then concluded by applying Proposition 5.1 and
(\ref{sinap}) and~(\ref{cdfid}).
\end{proof}
We now can close with the finite dimensional distribution of the
subordinator, which follows from Proposition 2.1 and Proposition
5.2.
\begin{thm}Consider the GGC$(1-\alpha,F_{1/\mathbb{G}_{\alpha}})$ subordinator
$(\Sigma_{\alpha}(t),t\leq 1/(1-\alpha)).$ Then for a disjoint
partition of the interval $(0,1/(1-\alpha)],$
$(C_{1},\ldots,C_{k}),$ the finite dimensional distribution of
$(\Sigma_{\alpha}(C_{1}),\ldots,\Sigma_{\alpha}(C_{k}))$ is such
that each $\Sigma_{\alpha}(C_{i})$ is independent with density
$$
\int_{0}^{\infty}{\mbox e}^{-x/w}w^{-1}f_{\alpha,\sigma_{i}}(w)dw
$$
where $\sigma_{i}=(1-\alpha)|C_{i}|$ and
$f_{\alpha,\sigma_{i}}(w)$ represents the density of
$M_{1}(F_{Y_{\sigma_{i}}/\mathbb{G}_{\alpha}})$ given in
Proposition 5.2. That is,
$$\Sigma_{\alpha}(C_{i})\overset{d}=G_{1}M_{1}(F_{Y_{\sigma_{i}}/\mathbb{G}_{\alpha}}).$$
\end{thm}

\end{document}